\documentclass[10pt]{article}

\usepackage[all]{xypic}
\usepackage{amsmath}
\usepackage{amsfonts}
\usepackage{amssymb}
\usepackage{amsthm}

\newcommand {\tb}{\textbf}
\newcommand {\mb}{\mathbb}
\newcommand {\bZ}{\mb Z}
\newcommand {\bR}{\mb R}
\newcommand {\bC}{\mb C}

\newcommand {\colim}{\textrm{colim}\ }
\newcommand {\cok}{\textrm{coker}}

\newcommand {\ex}{\mathrm{excess}}

\newtheorem{thm}{Theorem}[section]
\newtheorem{mthm}{Theorem}

\newtheorem{crl}[thm]{Corollary}

\newtheorem{lmm}[thm]{Lemma}
\newtheorem{mlmm}[mthm]{Lemma}

\newtheorem{exm}[thm]{Example}

\newtheorem{rmk}[thm]{Remark}

\begin{document}

\title{On spherical classes in $H_*QX$}
\author{Hadi Zare}
\date{}

\maketitle

\begin{abstract}
We give an upper bound on the set of spherical classes in $H_*QX$ when $X = P,S^1$.
This is related to the Curtis conjecture on spherical classes in $H_*Q_0S^0$. The results also provide some control over the bordism classes on of immersions when $X$ is a Thom complex.
\end{abstract}

\section{Introduction and statement of results}
Let $X$ be a CW-complex of finite type and let $Q_0X$ be the base-point component of $QX=\colim\Omega^i\Sigma^i X$. The Curtis conjecture predicts that only the Hopf invariant one and the Kervaire invariant one elements survive under the Hurewicz homomorphism $$h:{_2\pi_*^S}\simeq{_2\pi_*}Q_0S^0\to H_*Q_0S^0.$$
Here, and afterwards, we use ${_2\pi_*}$ for the $2$-primary homotopy, and $H_*$ for the $\bZ/2$-homology. A related problem is to determine the image of the Hurewicz homomorphism
$$h:{_2\pi_*^S}X\simeq {_2\pi_*}Q_0X\to H_*Q_0X$$
for $X$ path connected. The Curtis conjecture is related to this problem. For $X=P$, the infinite dimensional real projective space, the Kahn-Priddy theorem tells us that any spherical class in $H_*Q_0S^0$ pulls back to a spherical class in $H_*QP$. If $X=S^n$ the two problems are related throughout the homology suspension; a spherical class $\xi\in H_*Q_0S^0$ which is not decomposable survives to a spherical class $\sigma_*^k\xi\in H_{*+k}QS^k$ under iterated homology suspension $\sigma_*^k:H_*Q_0S^0\to H_{*+k}QS^k$ for some $k>0$. If $X=\Sigma\bC P_+$ the two problems are related by $S^1$-transfer.\\
While the Curtis conjecture is mostly motivated by stable homotopy theory, the calculation of the image of the Hurewicz homomorphism when $X$ is a Thom complex is also of interest in the bordism theory of immersions. The bordism class of a codimension $k$ immersion $i:M\to N$ with a $V$-structure on its normal bundle corresponds to a homotopy class $f_i:N_+\to QTV$ through the Pontrjagin-construction, where $V$ is a $k$-dimensional vector bundle. It is then well known that the normal Stiefel-Whitney numbers of such an immersion, and the immersions of its self-intersection manifolds, are determined by $(f_i)_*[N_+]$; the height $r$ elements in $(f_i)_*[N_+]$ correspond to the $r$-fold self intersection manifold \cite[Lemma 2.2]{0}. In particular, the case of framed codimension $k$ immersions in $\bR^{n+k}$ corresponds to spherical classes in $H_{n+k}QS^k$.\\ \\
It is possible to follow two distinct, but related approaches, in attempting to resolve the Curtis conjecture or more generally to attack the second problem. Firstly, we may try to calculate the image of $f\in{_2\pi_*}QX$ under $h$ provided that we a suitable description of $f\in{_2\pi_*^S}X\simeq{_2\pi_*}Q_0X$, such as a (stable) decomposition of $f$. Secondly, we may use basic properties of spherical classes to find an upper bound for the image of $h$, and then use homological methods to eliminate some of these classes from being in the image of $h$.\\
Following the first approach, we have shown that the Curtis conjecture is valid on ${_2\pi_*}J$ where $J$ is the fibre of $\psi^3-1:BSO\to BSO$ \cite[Theorem 1]{101}. Hence, in order to complete the proof of this conjecture, one only needs to calculate the image of $h:{_2\pi_*}Q_0S^0\to H_*Q_0S^0$ on ${_2\pi_*}\cok J$; or resolve the conjecture on some other sub-summands of ${_2\pi_*}\cok J$ in order to see a pattern. For instance, it is straightforward to show that $h\eta_i=0$ \cite{4} where $\eta_i\in{_2\pi_{2^i}^S}\simeq{_2\pi_{2^i}}Q_0S^0$ denotes Mahowald's family
\cite[Theorem 1]{10}. However, we do not have such information about every element in ${_2\pi_*^S}$. Nevertheless, the author believes that, in principle, one should be able to use the decompositions provided by results of Joel Cohen \cite{2} and apply methods similar to \cite{101} in order to obtain a proof of the Curtis conjecture on some other summands of ${_2\pi_*^S}$.\\
Indeed, a deeper insight to this conjecture, might be provided by observing that ${_2\pi_*}J$ does precisely correspond to that part of ${_2\pi_*^S}$ detected by $v_1$-periodic groups. Hence, assuming that one has some knowledge about $v_i$-periodic homotopy groups, and the infinite loop spaces associated with these groups,  may provide a deeper understanding of the problem.\\
The majority of the existing literature, attempting to prove the Curtis conjecture, follow the second approach including applications of algebraic transfer maps, Dickson invariants and etc. (see \cite{8} for a recent account). These approaches use unstable Adams spectral sequences, and lots of heavy algebraic machinery.\\
Our approach, however, is to use one basic property of spherical classes: a spherical class is a primitive $A$-annihilated class. Here, $\xi\in H_*X$ is called $A$-annihilated if $Sq^r_*\xi=0$ for $r>0$ where $Sq^r_*:H_*X\to H_{*-r}X$ is dual to $Sq^r$, $A$ being the mod $2$ Steenrod algebra. Therefore, the set of $A$-annihilated primitives determines an upper bound for the set of spherical classes. Our results provide some control over the set of all $A$-annihilated primitive classes in $H_*QX$ by giving some restriction on the form of these classes. Finding such an upper bound is a relatively easy task when $X$ is path connected since as Hopf algebras
\begin{equation}\label{eq1}
H_*QX\simeq\bZ/2[Q^Ix_\mu:\ex(Q^Ix_\mu)>0, I\textrm{ admissible}]
\end{equation}
where $\{x_\mu\}$ is an additive basis for $\widetilde{H}_*X$, $I=(i_1,\ldots,i_s)$ is a sequence of positive integers, called admissible if $i_j\leqslant 2i_{j+1}$ for $1\leqslant j\leqslant s-1$, and $\ex(Q^Ix_\mu)=i_1-(i_2+\cdots+i_s+\dim x_\mu)$. This description determines the $R$-module structure of $H_*QX$ where $R$ denote the mod $2$ Dyer-Lashof algebra (see \cite[Chapter 5]{15} for a detailed account on $R$). The $A$-module structure of $H_*QX$ is determined by (\ref{eq1}) together with the Nishida relations which we will recall in section $3$.\\
First, we determine the $A$-annihilated monomial generators $Q^Ix$. Define $\rho:\mathbb{N}\to\mathbb{N}$ by $\rho(n)=\min\{i:n_i=0\}$ for $n=\sum_{i=0}^\infty n_i2^i$ with $n_i\in\{0,1\}$. We have the following.

\begin{mthm}
Let $Q^Ix\in H_*QX$ with $I=(i_1,\ldots,i_s)$ admissible, and $\ex(Q^Ix)>0$. The class $Q^Ix$ is $A$-annihilated if and only if the following conditions are satisfied:\\
$(1)$ $x\in\widetilde{H}_*X$ is $A$-annihilated; $(2)$ $\ex(Q^Ix)<2^{\rho(i_1)}$; and $(3)$ $0\leqslant 2i_{j+1}-i_j<2^{\rho(i_{j+1})}$.\\
If $s=1$ then the first two conditions determine all $A$-annihilated classes of the form $Q^ix$ of positive excess.
\end{mthm}

The conditions in the above theorem ensure that if $Q^Ix$ is an $A$-annihilated class with $\ex(Q^Ix)>0$ then $I$ cannot have any even entry.\\ \\
Since the operations $Sq^r_*$ are additive homomorphisms, it is then clear that sum of $A$-annihilated terms is $A$-annihilated. However, the converse is not true, i.e. if $\xi=\sum Q^Ix+D$ is $A$-annihilated where $\ex(Q^Ix)>0$, and $D$ is a sum of decomposable terms, then it is not always true that $Q^Ix$ is $A$-annihilated. For instance, $Q^{2^t}a_{2^t-1}+a_{2^t}a_{2^t-1}\in H_*QP$ is annihilated by $Sq^1_*$ whereas $Q^{2^t}a_{2^t-1}$ is not. In particular, the class $Q^2a_1+a_1a_2$ is $A$-annihilated whereas $Q^2a_1$ is not $A$-annihilated. In order to deal with such cases, we consider applying the iterated homology suspension $\sigma_*^k:H_*QX\to H_{*+k}Q\Sigma^kX$. The fact that Steenrod operations are stable implies that if $\xi$ is $A$-annihilated then so is $\sigma_*^k\xi$ for all $k>0$. In section 2, we shall define an order on the monomial generators $Q^Ix\in H_*QX$ called the `total order'. The basic idea, then is that after finitely many times suspensions, terms of maximal total order will be left and these should be $A$-annihilated.\\
Given $I=(i_1,\ldots,i_s)$ define $I_j=(i_j,\ldots,i_s)$ for $1\leqslant j\leqslant s$. Our next result is as following.

\begin{mthm}\label{t2}
Let $X=P,S^1,\Sigma\bC P_+$ and let $\xi\in H_*QX$ be an $A$-annihilated class such that $\sigma_*\xi\neq 0$ where $\sigma_*:H_*QX\to H_{*+1}Q\Sigma X$ is the homology suspension. Then
$$\xi=Q^Ix+O$$
where $Q^{I_j}\Sigma^{n_j}x\in H_*Q\Sigma^{n_j}X$ is $A$-annihilated for some $j\leqslant s$, $n_j>0$, and $O$ denotes the sum of terms of lower total order and decomposable terms. Here, $\Sigma x$ is the image of $x$ under $\widetilde{H}_*X\simeq \widetilde{H}_{*+1}\Sigma X$.
\end{mthm}
A key fact in the proof of this theorem is that the spaces $S^1,P,\Sigma\bC P_+$ have the following property: it impossible to have $A$-annihilated classes of the form $Q^Ix+Q^Iy$, with $x\neq y$, where neither of $x,y$ is $A$-annihilated but $x+y$ is. The above theorem then is true for spaces like $P^n\times P^m$ or any arbitrary finite product of projective spaces and spheres. A simple example where the above theorem does not hold is obtained by taking $X$ to be the mapping cone of $\eta\vee3\eta:S^3\vee S^3\to S^2$ where to the eyes of $\bZ/2$-Steenrod operations the sum of classes at dimension $4$ is $A$-annihilated.\\
Next, we determine the type of terms $Q^Ix$ which can contribute to a primitive $A$-annihilated class $\xi\in H_*QX$.

\begin{mthm}\label{l1}
Let $\xi\in H_*Q_0X$ be an $A$-annihilated primitive class. The following statements hold.\\
$(1)$ If $\xi$ is even dimensional and $\xi=\sum Q^Ix$ modulo decomposable terms where $I$ only has odd entries.\\
$(2)$ If $\xi$ is an odd dimensional class and $X=\Sigma Y$ and $\xi=\sum Q^Ix$ modulo decomposable terms where $I$ only has odd entries.
\end{mthm}

We note that when $X$ is not a suspension, there is a counter example for $(2)$ in the above theorem. The class $Q^2a_1+a_1a_2+a_1^3+a_3\in H_3QP$ is an $A$-annihilated primitive class which does not satisfy the above theorem.\\ \\
If $\xi\in H_*Q_0X$ is an $A$-annihilated primitive class, so is $\xi^{2^t}$ for all $t$. Our final observation eliminates higher powers of $2$ when $\xi$ is spherical and $X$ is a suspension.
For a given space $X$ we write $r_X:H_*X\to H_*X$ for the square root map, dual to the cup squaring map $H^*X\to H^*X$ given by $x\mapsto x^2$.
\begin{mlmm}
Suppose $X=\Sigma Y$ such that $Q^Ix+Q^Jy\in\ker r$ for $x\neq y$ if and only if $Q^Ix,Q^Jy\in\ker r$ where $r:H_*QY\to H_*QY$ is the square root map. Then the following statements hold.\\
$(1)$ If $\xi=\zeta^{2^t}\in H_*QX$ is a spherical class with $\sigma_*\zeta\neq 0$ then $t\leqslant1$.\\
$(2)$ If $\xi=\zeta^2$ is spherical with $\sigma_*\zeta\neq 0$ then $\zeta$ is an odd dimensional class.
\end{mlmm}

Notice that the condition on $r$ is weaker than being monic, every class $Q^Ix$ with $I$ having at least one odd entry belongs to $\ker r$. In particular, the condition on $r$ is satisfied if $r_Y:H_*Y\to H_*Y$ is monic.\\ \\
Finally, we like to elaborate how these results might be applied. Suppose $X$ is path connected. First, recall that the height filtration $\mathrm{ht}:H_*QX\to\mathbb{N}$ is defined by $\mathrm{ht}(Q^i\xi)=2\mathrm{ht}(\xi)$, $\mathrm{ht}(\xi\xi')=\mathrm{ht}(\xi)+\mathrm{ht}(\xi')$, $\mathrm{ht}(x)=1$ for $x\in\widetilde{H}_*X$, $\xi,\xi'\in H_*QX$ \cite[Part I]{1}; in particular, $\mathrm{ht}(Q^Ix)=2^{l(I)}$. The height filtration is compatible with the Snaith splitting \cite{14} of $QX$, see also \cite{-1}, given by $\Sigma^\infty QX\simeq\vee_{r=1}^\infty D_rX$. The stability of the Steenrod operations implies that a class $\xi\in H_*QX$ is $A$-annihilated if sum of its terms of height $r$ is $A$-annihilated for all $r>0$.\\ Assume $\xi\in H_*QX$ is an $A$-annihilated primitive class, even dimensional for a moment, with $\sigma_*\xi\neq 0$. We may write $\xi=\sum Q^Ix+D$ where $D$ denotes the sum of decomposable terms, and $I$ has only odd entries. If $\xi=\sum_r(\sum_{\mathrm{ht}=r} Q^Ix+D_r)$ then each sum $\sum_{\mathrm{ht}=r} Q^Ix+D_r$ has to satisfy Theorem 2. For instance, if $\xi\in H_*QP$ is an $A$-annihilated primitive class with $\xi=\sum Q^Ix+D$ then
\begin{equation}\label{xi}
\xi=\sum_{l(I)\leqslant s}Q^Ia_{2^{t(I)}-1}+O
\end{equation}
where the first sum denotes the sum of terms of maximal total order of different lengths, $t(I)>0$, and $O$ denotes the sum of terms of lower total order and decomposable terms. Now, given that $\xi_0\in H_*Q_0S^0$ is a spherical class, this implies that it pulls back to a spherical class $\xi\in H_*QP$ which is of the form (\ref{xi}). The proof of Curtis conjecture then reduces to eliminating such classes from being spherical with $l(I)>0$. We hope that investigating these cases for the cases $X=\Sigma\bC P_+$ and some other spaces will also help in providing the Curtis conjecture. We leave this to a future work.\\
\tb{Acknowledgement.} This work has been done while I have been a long term visitor at the School of Mathematics/University of Manchester. I like to thank many individual for their support. Some of the results here are from my PhD thesis under supervision of Peter Eccles at Manchester. I am grateful to him for many discussions on the topic of this paper and for his encouragements during my PhD and the visiting time at Manchester. I  thank my family for their support.

\section{Proof of Theorem \ref{t2}}
\subsection{Kudo-Araki operations}
The Kudo-Araki operations $Q^i$ are group homomorphisms defined on homology of any infinite loop space $X$ as
$$Q^i:H_*X\to H_{*+i}X.$$
If $x\in H_*X$ is given with $\dim x=d$ then $Q^dx=x^2$ and $Q^ix=0$ for $i<d$ where the square is taken with respect to the Pontrjagin product in $H_*X$ given by the loop space structure on $X$. These operation also  satisfy various forms of Cartan formula \cite{1}. The iterated operation $Q^I$ also is defined by composition, i.e. $Q^{(I,J)}x=Q^IQ^Jx$ where $(I,J)=(i_1,\ldots,i_s,j_1,\ldots,j_t)$ for $I=(i_1,\ldots,i_s)$, $J=(j_1,\ldots,j_t)$. The equation (\ref{eq1}) completely determines the action of these operations on $H_*QX$ when $X$ is path-connected.\\
The Kudo-Araki operations with lower indexing $Q_i$ are defined by $Q_ix=Q^{i+d}x$ where $d=\dim x$, e.g. $Q_0x=x^2$. Using the notation of Theorem $2$ we may write $Q^Ix=Q_Ex$ where $E=(e_1,\ldots,e_s)$ with $e_j=\ex(Q^{I_j}x)$. The sequence $I$ is admissible if and only if $E$ is an increasing sequence; $I$ can not have even entries if and only if $E$ is strictly increasing.\\
We define the \textbf{total order} on terms $Q^Ix\in H_*QX$ as following. For $E$ and $E'$ sequences of nondecreasing integers, of the same length say $s$, we define $E>E'$ either if $e_1>e'_1$, or if $e_j=e'_j$ for all $j<t$ with some $t\leqslant s$ and $e_t>e'_t$, which is the same as the lexicographic ordering. We then rearrange terms of the form $Q^Ix$ by $Q^Ix=Q_Ex>Q_{E'}y=Q^Jy$ if $E>E'$. We call this the `\textbf{total order}'. We note it is possible to define the total order on classes $Q^Ix_i\in H_*Q_0S^0$ in a similar way.

The homology suspension homomorphism $\sigma_*:H_*QX\to H_{*+1}Q\Sigma X$ kills decomposable terms while $\sigma_*Q^Ix=Q^I\Sigma x$. Notice that if $\ex(Q^Ix)=0$ then it is a square and hence a decomposable term. This completely determines the action of $\sigma_*$. In terms of lower indexed operations, if $Q^Ix=Q_Ex$ then
$$\sigma_*Q_Ex=Q_{E-\mathbf{1}}\Sigma x$$
where for $n>0$ we have $\mathbf{n}=(n,n,\ldots,n)$ and $E-\mathbf{n}$ is defined componentwise. It is obvious that the (iterated) homology suspension respects the total order.
\subsection{Proof of Theorem \ref{t2}}
Suppose $\xi\in H_*QX$ is $A$-annihilated with $\sigma_*\xi\neq 0$. We may write $\xi=\sum Q^Ix+D$ where $\ex(Q^Ix)>0$ and $D$ denotes the sum of decomposable terms. Assume that $X$ has only one cell in each dimension, without loss of generality one can take $X=P$. It is then immediate that there is a unique term of maximal total order, that is we may write
$$\xi=Q^Ix+O$$
where $Q^Ix=Q_Ex$ is the term of maximal total order, with $E=(e_1,\ldots,e_s)$, and $O$ denotes sum of terms of lower order and decomposable terms.\\
\textit{Case} 0: Suppose $Q^Ix$ is the only term with $\ex(Q^Ix)=e_1$ and other terms are of lower excess.\\
If $e_1>1$ then $\sigma_*^{e_1-1}\xi=Q^I\Sigma^{e_1-1}x$ which is an $A$-annihilated class. In this case, applying Theorem 1 to this class completes the proof.\\
If $e_1=1$ then all other terms in the expression for $\xi$ are decomposable. In this case, $\sigma_*\xi=(Q^{I_j}\Sigma x)^{2^t}$ for some $j\leqslant s$ and some $t$. The class $(Q^{I_j}\Sigma x)^{2^t}$ is an $A$-annihilated class which implies that $Q^{I_j}\Sigma x$ is $A$-annihilated and the proof is complete.\\
Next, assume $Q^Ix$ is not the only term of maximum excess. Let $Q^Jy=Q_{E'}y$ be the next term in the expression for $\xi$ with respect to the total order. By definition of the total order, since $E>E'$ with $e_1=e_1'$ then there exists $t\leqslant s$ such that $e_j=e_j'$ for all $j<t$ and $e_t>e_t'$. Let us divide the proof to the following cases:\\
\textit{Case} 1-1: $e_1>1$\\
If $\xi$ even dimensional then Corollary \ref{c1} and Lemma \ref{55} imply that in all terms $Q^I x$ of positive excess, $I$ can have only odd entries. If $\xi$ is odd dimensional then for similar reason in all terms $Q^I\Sigma x$ in $\sigma_*\xi$ which are of positive excess, $I$ will have only odd entries. In particular $Q^Ix$ and $Q^Jy$ must have this property. Consequently, all sequences $E$ in the expression for $\sigma_*\xi$ with lower indexed operations should be strictly increasing. This insures that $\sigma_*^{e_j}Q^{I_j}x=(Q^{I_{j+1}}\Sigma^{e_j}x)^2$ with $\ex(Q^{I_{j+1}}\Sigma^{e_j}x)>0$, i.e. $Q^{I_{j+1}}\Sigma^{e_j}x\neq 0$.\\
Assume $s=2$. In this case
$$\sigma_*^{e_1}\xi=(Q^{I_2}\Sigma^{e_1}x+Q^{J_2}\Sigma^{e_1}y+O_2)^2.$$
Reset $\xi:=Q^{I_2}\Sigma^{e_1}x+Q^{J_2}\Sigma^{e_2}y+O_2$ whose leading term with respect to the total order is $Q^{I_2}\Sigma^{e_1}x$ and is a term of maximum of excess. Now applying \textit{Case }0 to this class completes the proof. For $s>2$ iterated application of this argument completes the proof.\\
\textit{Case} 1-2: $e_1=1$\\
In this case, $\sigma_*\xi=(Q^{I_j}\Sigma x+Q^{J_j}\Sigma y+O_j)^{2^t}$ is annihilated class where $j\leqslant s$ and $t>0$. This implies that $Q^{I_j}\Sigma x+Q^{J_j}\Sigma y+O_j$ is $A$-annihilated. Reset $\xi:=Q^{I_j}\Sigma x+Q^{J_j}\Sigma y+O_j$ and start from \textit{Case 0}.\\
The above algorithm depends on the length of $I$. Hence, it will terminate in finitely many steps. This completes the proof of Theorem 2.\\ \\
Notice that according to the Nishida relations, to be explained in next session, the action of the Steenrod operation respects the length. Hence, it is enough to focus on terms of fixed length.

\section{The Nishida and Adem relations}
The action of the Steenrod algebra, $A$, on $H_*QX$ is described by its actions on the monomial generators $Q^Ix$ which is provided by the Nishida relation \cite[Part I, Theorem 1.1(9)]{1}
\begin{equation}
Sq^a_*Q^b=\sum_{t\geqslant 0}{b-a\choose a-2t}Q^{b-a+r}Sq^t_*.
\end{equation}
The iterated application of this relation allows us to calculate $Sq^a_*Q^I$ when $l(I)>1$. Notice that \textit{the Nishida relations respect the length}, i.e. if
\begin{equation}\label{n1}
Sq^a_*Q^I=\sum Q^KSq^{a^K}_*
\end{equation}
then $l(I)=l(K)$. According to Madsen \cite[Equation 3.2]{9} the Nishida relations define an action $N:A\otimes R\to R$ as following
\begin{eqnarray}
N(Sq^a_*,Q^b)&=&{b-a\choose a}Q^{b-a},\\
N(Sq^a_*,Q^I)& = & \sum{i_1-a\choose a-2t}Q^{i_1-a+t}N(Sq^t_*,Q^{I_2}).
\end{eqnarray}
In other words if $Sq^a_*Q^I=\sum Q^KSq^{a^K}_*$ where $a^K\in\bZ$ then
$$N(Sq^a_*,Q^I)=\sum_{a^K=0}Q^K.$$
We are concerned about vanishing or non-vanishing of these relations which usually can be determined by looking at the binomial coefficients (mod $2$) in a given relation. Recall that if $n=\sum n_i2^i$ and $m=\sum m_i2^i$ are given with $n_i,m_i\in\{0,1\}$ then ${n\choose m}=1$ mod $2$ if and only if $n_i\geqslant m_i$ for all $i$. For instance, the equality ${\mathrm{even}\choose\mathrm{odd}}=0$ implies that $Sq^{2a+1}_*Q^{2b+1}=0$ and more generally $Sq^{2a+1}_*Q^I=0$ if $I=(2b+1,i_2,\ldots,i_s)$.\\
However, it is not always very straightforward to decide about triviality or non-triviality of $Sq^a_*Q^I$ as in the Nishida relations \ref{n1} it is not clear that the sequences $K$ are admissible. We need to rewrite \ref{n1} as a sum of terms where $K$ is admissible and then decide about vanishing of $Sq^a_*Q^I$.\\
If $Q^aQ^b$ is non admissible, i.e. $a>2b$, then the \textit{Adem relation} is given by
\begin{equation}
Q^aQ^b=\sum_{a+b\leqslant 3t}{t-b-1\choose 2t-a}Q^{a+b-t}Q^t
\end{equation}
allows us to rewrite $Q^aQ^b$ as a sum of admissible terms.

\begin{exm}\label{exm}
Consider $Q^9Q^5g_1$ which is an admissible term. One has
$$Sq^4_*Q^9Q^5g_1=Q^7Q^3g_1$$
where $Q^7Q^3$ is not admissible. In fact the Adem relation $Q^7Q^3=0$ implies $Q^7Q^3g_1=0$. Indeed the class $Q^9Q^5g_1$ is not $A_*$-annihilated which can be seen by applying $Sq^2_*$ as we have
$$Sq^2_*Q^9Q^5g_1=Q^7Q^5g_1\neq 0.$$
Notice that the right hand side of the above equation is an admissible term.
\end{exm}

We need to choose the suitable operation $Sq^a_*$ such that the outcome is admissible and there is no need to use the Adem relations for the Kudo-Araki operation. The reason is that it is practically impossible to use the Adem relations when $l(I)$ is big. The following lemma of Curtis \cite[Lemma 6.2]{3} tells us when it is possible to choose the right operation and provides us with the main tool towards the proof of Theorem 1.

\begin{lmm}\label{curtis}
Suppose $I=(i_1,\ldots,i_s)$ is an admissible sequence such that $2i_{j+1}-i_j< 2^{\rho(i_{j+1})}$ for all $1\leqslant j\leqslant s-1$. Let
$$N(Sq^a_*,Q^I)=\sum_{K\textrm{ admissible}} Q^K.$$
Then
$$\ex(K)\leqslant \ex(I)-2^{\rho(i_1)}.$$
Moreover, $\rho(i_1)\leqslant\rho(i_2)\leqslant\cdots\leqslant\rho(i_s)$.
\end{lmm}

The above lemma can also be obtained by combining observations of Wellington \cite[Theorem 7.11]{15}, \cite[Theorem 7.12]{15} and \cite[Lemma 12.5]{15}.

We conclude this section by some further remarks.
\begin{rmk}\label{relations}
We make some simple observations about the Adem and Nishida relations.\\
$(0)$ Applying the Adem relations and Nishida relations reduces the excess.\\
$(1)$ The Nishida relation $Sq^\mathrm{odd}_*Q^\mathrm{odd}=0$ implies that if we have a sequence $I$ only with odd entries, then after iterated application of the Nishida relations, and before applying the Adem relations to non-admissible terms, at the expression
$$Sq^{2k}_*Q^I=\sum Q^KSq^{a^K}_*$$
the sequence $K$ will only have odd entries.\\
$(2)$ The binomial coefficient in the Adem relation tell us that for any $a,b\geqslant 0$ we have
\begin{equation}
Q^{2a+1}Q^{2b+1}=\sum \epsilon_t Q^{\mathrm{odd}}Q^{\mathrm{odd}}
\end{equation}
with $\epsilon_t\in\{0,1\}$. This fact together with (1) implies that if $I$ has only odd entries, applying the Adem relations to the non-admissible terms we may write
$$Sq^{2k}_*Q^I=\sum Q^{L_K}Sq^{a^K}_*$$
with $L_K$ admissible only having odd entries.\\
$(3)$ The binomial coefficients in the Adem relation yield
\begin{equation}
Q^{2a+1}Q^{2b}=\sum \epsilon_t Q^{\mathrm{odd}}Q^{\mathrm{even}}.
\end{equation}
This implies that if $i$ is an even number, and $I$ is a sequence of odd numbers then
$$Sq^{2^{r-1}}_*Q^{(I,i)}=\sum Q^LSq^{a^L}_*$$
with $L=(l_1,\ldots,l_s,l_{s+1})$ admissible where $l_1,\ldots,l_s$ are odd and $l_{s+1}$ is even.\\
$(4)$ The binary coefficient of $Q^{a+b-t}Q^t$ in the Adem relation for a non-admissible term $Q^aQ^b$ can be nontrivial if $t>b$. Writing in the lower indexed operations, this can be used to show that applying the Adem relation to non-admissible pairs $Q^aQ^b$ reduces the total order, that is if $Q^aQ^b$ is not admissible then $Q^aQ^bx>Q^{a+b-t}Q^tx$ for all $t>b$.
\end{rmk}

\section{Proof of Theorem $1$}
We break the proof into little lemmata.

\begin{lmm}
Let $x\in H_*X$ be $A$-annihilated, and $I$ an admissible sequence such that $(1)$ $0<\ex(Q^Ix)<2^{\rho(i_1)}$; $(2)$ $2i_{j+1}-i_j< 2^{\rho(i_{j+1})}$ for all $1\leqslant j\leqslant s-1$. Then $Q^Ix$ is $A$-annihilated.
\end{lmm}

\begin{proof}
Let $a>0$. Since $x$ is $A$-annihilated, we have
$$\begin{array}{llllll}
Sq^a_*Q^Ix & = & \sum Q^KSq^{a^K}_*x & = & \sum_{a^K=0} Q^Kx
           \end{array}$$
where $K$ is admissible. But notice that according to Lemma 2.9
$$\ex(Q^Kx)\leqslant\ex(Q^Ix)-2^{\rho(i_1)}<0.$$
Hence the above sum is trivial, and we are done.
\end{proof}

This proves the Theorem $1$ in one direction. In order to prove the other direction,  we show if any of the conditions (1)-(3) of Theorem $1$ does not hold then $Q^Ix$ will be not-$A$-annihilated.

\begin{rmk}\label{ro}
Looking at the binary expansions it is easy to see that given a positive integer $n$, then $\rho(n)$ is the least integer $t$ such that
$${n-2^t\choose 2^t}\equiv 1 \textrm{ mod }2.$$
\end{rmk}

\begin{lmm}
Let $X$ be path connected. Suppose $I=(i_1,\ldots,i_s)$ is an admissible sequence, and let $Q^Ix$ be given with $\ex(Q^Ix)>0$ with $j$ being the least positive integer such that $2i_{j+1}-i_j\geqslant 2^{\rho(i_{j+1})}$. Then such a class is not $A$-annihilated, and we have
$$Sq^{2^{\rho(i_{j+1})+j}}_*Q^Ix=Q^{i_1-2^{\rho+j-1}}Q^{i_2-2^{\rho+j-2}}\cdots Q^{i_j-2^\rho}Q^{i_{j+1}-2^\rho}Q^{i_{j+2}}\cdots Q^{i_s}x$$
modulo terms of lower excess and total order.
\end{lmm}

\begin{proof}
Assume that $Q^Ix$ satisfies the condition above. We may write this
condition as
$$i_j-2^\rho\leqslant 2i_{j+1}-2^{\rho+1}=2(i_{j+1}-2^\rho),$$
where $\rho=\rho(i_{j+1})$. This is the same as the admissibility condition for the pair $(i_j-2^\rho,i_{j+1}-2^\rho)$. In this case we use $Sq^{2^{\rho+j}}_*$ where we get
$$\begin{array}{lll}
Sq^{2^{\rho+j}}_*Q^Ix&=&\underbrace{Q^{i_1-2^{\rho+j-1}}Q^{i_2-2^{\rho+j-2}}\cdots
Q^{i_j-2^\rho}Q^{i_{j+1}-2^\rho}Q^{i_{j+2}}\cdots Q^{i_s}x}_A+O
\end{array}$$
where $O$ denotes other terms which according to Remark \ref{relations} is a sum of terms of lower excess and lower total order. The term $A$ in right hand side of the of the above equality is admissible. Moreover,
$$\begin{array}{lll}
\ex(A) & = & (i_1-2^{\rho+j-1})-(i_2-2^{\rho+j-2})-(i_j-2^\rho)-(i_{j+1}-2^\rho)-\\
                           &   & (i_{j+2}+\cdots+i_s+\dim x)\\
                           & = & i_1-(i_2+\cdots+i_s+\dim x)\\
                           & = & \ex(Q^Ix)>0.
\end{array}$$
First, this implies that $A$ is nontrivial. Second, being of higher excess and total order shows that $A$ will not be equal to any of terms in $O$. This implies that $Sq^{2^{\rho+j}}_*Q^Ix\neq 0$ and hence completes the proof.
\end{proof}

Notice that choosing the least $j$ is necessary, as otherwise we may not get nontrivial action (see Example \ref{exm}). Now, assume that the above condition does hold, but condition (2) in Theorem 1 fails. This case is resolved in the following theorem.

\begin{lmm}
Let $X$ be path connected. Suppose $I=(i_1,\ldots,i_s)$ is an admissible sequence, such that $\ex(Q^Ix)\geqslant 2^{\rho(i_1)}$, and $2i_{j+1}-i_j<2^{\rho(i_{j+1})}$ for all $1\leqslant j\leqslant s-1$. Then such a class is not $A$-annihilated.
\end{lmm}

\begin{proof}
We use $Sq^{2^\rho}_*$ with $\rho=\rho(i_1)$ which gives
$$\begin{array}{lll}
Sq^{2^\rho}_*Q^Ix & = & Q^{i_1-2^\rho}Q^{i_2}\cdots Q^{i_s}x+O
\end{array}$$
where $O$ denotes other terms given by
$$O=\sum_{t>0}{i_1-2^\rho\choose
2^\rho-2t}Q^{i_1-2^\rho+t}Sq^t_*Q^{i_2}\cdots Q^{i_s}x.$$
Notice that $\ex(Q^Ix)\geqslant 2^{\rho(i_1)}$ ensures that $i_1$ is not of the form $2^\rho$. By iterated application of the Nishida relations we may write
$$O=\sum_{\alpha\leqslant s} \epsilon_1\cdots \epsilon_\alpha Q^{i_1-2^\rho+r_1}Q^{i_2-r_2+r_3}\cdots Q^{i_\alpha-r_\alpha}Q^{\i_{\alpha+1}}\cdots Q^{i_s}x$$
where
$$\epsilon_1={i_1-2^\rho\choose 2^\rho-2r_1},\textrm{ }\epsilon_2={i_2-r_1\choose r_1-2r_2},\ldots,\epsilon_{\alpha-1}={i_{\alpha-1}-r_{\alpha-2}\choose 2r_{\alpha-2}-2r_{\alpha-1}}, \textrm{ }
\epsilon_\alpha={i_\alpha-r_{\alpha-1}\choose r_{\alpha-1}}$$
such that $2r_k\leqslant r_{k-1}$ for all $k\leqslant\alpha$. The sequence $I$ satisfies the conditions of Lemma \ref{curtis} which in particular implies that $\rho(i_1)\leqslant\cdots\leqslant \rho(i_\alpha)\leqslant\cdots\leqslant\rho(i_s)$. Notice that $r_{\alpha-1}<2^{\rho(i_\alpha)-\alpha+1}<2^{\rho(i_\alpha)}$ which together with Remark \ref{ro} implies that $\epsilon_\alpha=0$ and therefore $O=0$. This then shows that
$$Sq^{2^\rho}_*Q^Ix  =  Q^{i_1-2^\rho}Q^{i_2}\cdots Q^{i_s}x\neq 0.$$
This completes the proof.
\end{proof}

Now we show that the condition (1) is also necessary in the proof of the Theorem 1.2

\begin{lmm}
Let $X$ be path connected, and let $Q^Ix\in H_*QX$ be a term of positive excess with  $I$ admissible such that $x\in \widetilde{H}_*X$ is not $A$-annihilated. Then $Q^Ix$ is not $A$-annihilated.
\end{lmm}

\begin{proof}
Let $t$ be the least number that $Sq^{2^s}_*x\neq 0$. If $I=(i_1,\ldots,i_s)$ we apply $Sq^{2^{s+t}}_*$ to $Q^Ix$ which gives
we get
$$\begin{array}{lll}
Sq^{2^{s+t}}_*Q^Ix & = & Q^{i_1-2^{s+t-1}}\cdots Q^{i_s-2^s}Sq^{2^s}_*x+O,
\end{array}$$
where $O$ denotes sum of the other terms which are of the form $Q^Jx$. This means that the first term in the above equality will not cancel with any of the other terms. Notice that the first term in the above expression is admissible, and $\ex(Q^{i_1-2^{s+t-1}}\cdots Q^{i_t-2^s}Sq^{2^s}_*x)=\ex(Q^Ix)>0$. Hence $Sq^{2^{s+t}}_*Q^Ix\neq 0$.
\end{proof}

\section{Proof of Theorem \ref{l1}}
\begin{lmm}\label{lemma}
Assume $X$ is path connected and let $\xi\in H_*QX$ be given by
$$\xi=\sum Q^Ix+D$$
where $I=(i_1,\ldots,i_s)$, $r>0$, varies over certain admissible terms, $\ex(Q^Ix)>0$, and $D$ denotes decomposable terms. Suppose $\xi$ is $A$-annihilated. Then $i_1$ is odd for all terms $Q^Ix$ with $\ex(Q^Ix)\geqslant 2$ .
\end{lmm}
Before proceeding with proof, recall the Nishida relations $Sq^1_*Q^{2t}=Q^{2t-1}$ and $Sq^1_*Q^{2t-1}=0$. Moreover, we note that we have the Cartan formula $Sq^a_*(xy)=\sum_r (Sq^r_*x)(Sq^{a-r}_*y)$ for all $x,y\in H_*\Omega X$ \cite{15}.
\begin{proof}
Proof is by contradiction. Assume there is a term $Q^Ix$ in $\xi$ with $i_1$ even and $\ex(Q^Ix)\geqslant2$.
Applying the Nishida relations above we obtain
$$Sq^1_*\xi=\sum_{i_1\textrm{ even}} Q^{i_1-1}Q^{i_2}\cdots Q^{i_s}x+Sq^1_*D.$$
According to the Cartan formula $Sq^1_*D$ is a sum of decomposable terms. Moreover, since the Nishida relations respect length we then focus on terms $Q^Ix$ where $l(I)=s$ is fixed. If $Q^Ix$ is a term with $\ex(Q^Ix)\geqslant 2$ and $i_1$ even, then $Sq^1_*Q^Ix$ is of positive excess, i.e. it is not decomposable. If $Q^Jy$ is another term with $\ex(Q^Jy)\geqslant2$, $j_1$ even, and $Q^Ix\neq Q^Jy$ then it is clear that $Sq^1_*(Q^Ix+Q^Jy)\neq 0$. This shows that the first sum in the above expression for $Sq^1_*\xi$ is nontrivial. This contradicts the fact that $\xi$ is $A$-annihilated. This completes the proof.
\end{proof}

\begin{crl}\label{c1}
$(1)$ Let $\xi\in H_*QX$ be an even dimensional $A$-annihilated class with $\sigma_*\xi\neq 0$ where $\sigma_*:H_*QX\to H_*Q\Sigma X$ is the homology suspension. Then $\xi=\sum Q^Ix$ modulo decomposable terms such that $i_1$ is odd.\\
$(2)$ Let $\xi$ be an odd dimensional $A$-annihilated class. Then $\xi=\sum Q^Ix$ modulo decomposable terms, with $\ex(Q^Ix)>0$, and $i_1$ odd if $\ex(Q^Ix)\geqslant 3$.
\end{crl}

\begin{proof}
$(1)$ Since $\sigma_*\xi\neq0$ hence $\xi=\sum Q^Ix+D$ where $D$ denotes sum of decomposable terms. Notice that $\ex(Q^Ix)$ and $\dim Q^Ix$ have the same parity. Since $Q^Ix$ is even dimensional and $\ex(Q^Ix)>0$ hence $\ex(Q^Ix)\geqslant 2$ and we are done.\\
$(2)$ Since $Q^Ix$ is odd dimensional then $\ex(Q^Ix)$ is also odd. Applying previous lemma shows that $i_1$ is odd if $\ex(Q^Ix)\geqslant 3$.
\end{proof}

Notice that we have counter example for the cases of odd dimensional classes with terms $\ex(Q^Ix)=1$ e.g. the class $Q^{2^t}a_{2^t-1}+a_{2^t}a_{2^t-1}\in H_*QP$ is $Sq^1_*$-annihilated but $Q^{2^t}a_{2^t-1}$ does not satisfy the above theorem. The restriction on the odd dimensional classes can be lifted if we restrict to $A$-annihilated primitive classes with $X=\Sigma Y$.
\begin{lmm}\label{53}
Let $\xi\in H_*QX$ be an odd dimensional $A$-annihilated primitive class where $X=\Sigma Y$. Then $\xi=\sum Q^Ix$ with $\ex(Q^Ix)>0$ and $i_1$ odd.
\end{lmm}

\begin{proof}
Since $\xi$ is odd dimensional and $X$ is a suspension, we may write $\xi=\sum Q^I\Sigma y+D$ where $D$ denotes the sum of decomposable terms. Notice that each term $Q^I\Sigma y$ is primitive, hence $D$ must be a primitive class. A decomposable primitive class should be a square. This contradicts the fact that $\xi$ is odd dimensional. Hence, $D=0$.\\
Without loss of generality assume all terms are of excess equal to $1$. Then,
$$Sq^1_*\xi=(\sum_{i_1\textrm{ even}} Q^{I_2}x)^2\neq 0.$$
But this contradicts the fact that $\xi$ is $A$-annihilated. This completes the proof.
\end{proof}

Note that in particular, every spherical class $\xi\in H_*QS^n$, $n>0$, satisfies the above lemma.\\

Given two sequence $I,J$ of the same length, we write $I-J$ for the obvious subtraction defined component-wise. We write $2|I$ if $2|i_j$ for all $1\leqslant j\leqslant s$. We have the following.

\begin{lmm}\label{55}
Let $X$ be path connected, and let $\xi=\sum Q^Ix+D\in H_*QX$ be an $A$-annihilated class with $i_1$ odd for all $I$. Then $I$ only consists of odd entries for all $I$.
\end{lmm}

\begin{proof}
The proof is by induction. Suppose $I=(i_1,\ldots,i_l)$ is given with first $s$ entries odd numbers. Applying $Sq^{2^s}_*$ to $Q^Ix$, and before applying possible Adem relations to non-admissible terms, we obtain
$$\begin{array}{lll}
Sq^{2^s}_*Q^Ix & = & \epsilon Q^{i_1-2^{s-1}}Q^{i_2-2^{s-2}}\cdots Q^{i_s-1}Q^{i_{s+1}-1}Q^{I_{s+2}}x +\\
               &   & \sum_{M\in A_{(I,x)}} \epsilon_M Q^{i_1-2^{s-1}}Q^{I_{2,s}-M}Q^{I_{s+1}}x + \\
               &   & \sum_{N\in B_{(I,x)},t<2^{s-1}} \epsilon_N Q^{i_1-t}Q^{I_{2,s}-N}Q^{I_{s+1}}x
               \end{array}$$
where $I_{2,s}=(i_2,\ldots,i_s)$, $A_{(I,x)}\subset\{M=(m_1,\ldots,m_{s-1}):2|M, \dim M=2^s-1\}$, $B_{(I,x)}\subset\{N=(n_1,\ldots,n_{s-1}):2|N,\dim N=2^s-t\}$. Here, $\epsilon=1$ if $i_{s+1}$ is even, and $\epsilon=0$ otherwise; $\epsilon_M$ and $\epsilon_N$ are the coefficients determined by iterated applications of the Adem relations. Notice that the first term is admissible, and is nontrivial if $i_{s+1}$ is odd. In this case, the $(s+1)$-th entry of the sequence $({i_1-2^{s-1}},{i_2-2^{s-2}}\ldots {i_s-1},{i_{s+1}-1},{I_{s+2}})$ is $i_{s+1}-1$ which is an odd number. However, according Remark \ref{relations} applying the Adem relations to the non-admissible terms in the next two sums will result in terms whose $(s+1)$-th entry is an even number. This implies that none of the terms in these sums will cancel with the first term of the expression for $Sq^{2^s}_*Q^Ix$ if $i_{s+1}$ is even. Moreover, notice that every single term in $\sum_{N\in B,t<2^{s-1}} \epsilon_N Q^{i_1-t}Q^{I_{2,s}-N}Q^{I_{s+1}}x$ is of excess strictly less than $\ex(Q^Ix)$ whereas
$$\ex(Q^{i_1-2^{s-1}}Q^{i_2-2^{s-2}}\cdots Q^{i_s-1}Q^{i_{s+1}-1}Q^{I_{s+2}}x)=\ex(Q^Ix).$$
The terms of the first sum are of excess at most equal to $\ex(Q^Ix)$.\\

Now assume $\xi=\sum Q^Ix$ modulo decomposable terms where $i_1$ is odd, $x\in H_*X$, and $I$ varies over certain admissible terms. Assume that we have shown the first $s$ entry of all sequences of length $r$ are odd. We like to show that the $(s+1)$-th entry is also odd. Applying $Sq^{2^s}_*$ we obtain
$$\begin{array}{lll}
Sq^{2^s}_*\xi  & = & \sum_{i_{s+1}\textrm{ even}} Q^{i_1-2^{s-1}}Q^{i_2-2^{s-2}}\cdots Q^{i_s-1}Q^{i_{s+1}-1}Q^{I_{s+2}}x +\\
               &   & \sum_{A_{(I,x)}}\sum_{M\in A} \epsilon_M Q^{i_1-2^{s-1}}Q^{I_{2,s}-M}Q^{I_{s+1}}x + \\
               &   & \sum_{B_{(I,x)}}\sum_{N\in B,t<2^{s-1}} \epsilon_N Q^{i_1-t}Q^{I_{2,s}-N}Q^{I_{s+1}}x.
\end{array}$$
First, note that non of terms in the first term will cancel with any of terms in the other two sums, because of existing an odd entry at $(s+1)$-th entry. Moreover, if we have two distinct terms, $Q^Ix$ and $Q^IJy$ in the expression for $\xi$ then it is clear that
$$Q^{i_1-2^{s-1}}Q^{i_2-2^{s-2}}\cdots Q^{i_s-1}Q^{i_{s+1}-1}Q^{I_{s+2}}x\neq Q^{j_1-2^{s-1}}Q^{j_2-2^{s-2}}\cdots Q^{j_s-1}Q^{j_{s+1}-1}Q^{J_{s+2}}y.$$
This implies that $Sq^{2^s}_*\xi\neq 0$ which is a contradiction. Hence, the lemma is proved.
\end{proof}

The above lemma together with Lemma \ref{53} and Corollary \ref{c1} completes the proof of Theorem 3.

\section{Proof of Lemma $4$}
First, we recall the following. The fact that $r$ is dual to the cup-squaring map can be used to show that
\begin{eqnarray}\label{root}
rQ^{2i}\xi=Q^ir\xi, &  rQ^{2i+1}=0.
\end{eqnarray}
This implies that any term with $Q^Ix$ with $I$ having at least one odd entry belongs to the kernel of $r$.

We prove the statement $(1)$. The proof of $(2)$ is similar.\\

A spherical class $\xi\in H_{*+1}Q\Sigma Y$ pulls back to a spherical class $\xi_0\in H_*Q_0Y$ where $Q_0Y$ denotes the base-point component of $QY$. For simplicity, we assume $Y$ is path connected so that $Q_0Y=QY$ and assume $t=2$, i.e. $\xi=\zeta^4=Q^{2d}Q^d\zeta$. The cases $t>2$ are similar. The class $\xi$ can be primitive only if $\zeta$ is. The Cartan formula $Sq^{2r}_*\zeta^2=(Sq^r_*\zeta)^2$ implies that $\zeta$ is $A$-annihilated. Since $\sigma_*\zeta\neq 0$ follows that $\zeta=\sum Q^I\Sigma y+D$ where $D=\sum Q^J\Sigma y$ with $\ex(Q^J\Sigma y)=0$. The class $\xi$ pulls back to a spherical class
$$\xi_0=\sum Q^{2d}Q^dQ^Iy+\sum Q^{2d}Q^dQ^Jy=Q^{2d}Q^d(\sum Q^Iy+\sum Q^Jy)=\sum Q^{2d}Q^dQ^Ky$$
modulo decomposable terms, where $K$ takes all $I$'s and $J$'s into account. The class $\xi_0$ is an $A$-annihilated primitive class. The primitivity implies that the above some, representing the nondecomposable part of $\xi_0$ must belong to $\ker r$ \cite[Proposition 4.23]{12}. Since, $r_Y$ is a monic it follows that $r:H_*Q_0Y\to H_*Q_0Y$ satisfies the required condition in Theorem 4, i.e. $r$ maps $\sum Q^{2d}Q^dQ^Ky$ trivially if it maps each term $Q^{2d}Q^dQ^Ky$ trivially. Hence, $d$ must be odd or $K$ has to have at least one odd entry or $y$ belongs to $\ker r_Y$. Each class $Q^dQ^Ky$ corresponds to a unique primitive $P_{K,y}$ modulo decomposable terms. That is we may write
$$\xi=Q^{2d} \sum P_{K,y}+D$$
where $D$ denotes sum of decomposable terms, which is also ought to be primitive. Hence, $D=0$. Applying $Sq^1_*$ we obtain
$$Sq^1_*\xi=(\sum P_{K,y})^2\neq 0.$$
This contradicts the fact that $\xi_0$ is $A$-annihilated, and completes the proof.

I have been an independent mathematician at the School of Mathematics/The University of Manchester since May 2009 and I am grateful to many individual in the School for their support as well as for the hospitality of the School. I am grateful to Peter Eccles who read the draft and made many helpful suggestion. Some of the results in this paper are from my PhD \cite{100} under supervision of Peter Eccles at the University of Manchester. I am very grateful to him for his unlimited encouragement/mathematical discussions since the beginning of my PhD which still continues till today. I was partly supported by the School of Mathematics/The University of Manchester during my PhD and I am very grateful for this support. Finally, I like to thank my family, especially my parents, for their full support.

\end{document}